\documentclass[]{article}

\usepackage{microtype,amssymb,amsmath,enumitem,tikz,amsthm,hyperref,cleveref,xcolor,nicefrac,soul}
\usetikzlibrary{angles,quotes}
\hypersetup{
	colorlinks,
	linkcolor={red},
	citecolor={red},
	urlcolor={blue}
}

\newcommand{\bR}{\mathbb{R}}

\newcommand{\ep}{\ensuremath{\varepsilon}}
\newcommand{\vp}{\ensuremath{\varphi}}
\newcommand{\n}[1]{\ensuremath{||#1||}}

\newtheorem{thm}{Theorem}[section]
\newtheorem{lma}{Lemma}[section]
\newtheorem{dfn}{Definition}[section]

\begin{document}

\title{A New Construction of Forests with Low Visibility}
\author{Kirill Kashkan\footnote{partly funded by the Deutsche Forschungsgemeinschaft (DFG, German Research Foundation) under Germany's Excellence Strategy – EXC-2047/1 – 390685813 and by NSERC through the author's supervisor, Almut Burchard.}}
\date{}

\maketitle

\begin{abstract}
A set of points with finite density is constructed in $\bR^d$, with $d\geq2$, by adding points to a Poisson process such that any line segment of length $O\left(\ep^{-(d-1)}\ln\ep^{-1}\right)$ in $\bR^d$ will contain one of the points of the set within distance $\ep$ of it. The constant implied by the big-$O$ notation depends on the dimension only.
\end{abstract}

\section{Introduction}

Consider a forest where all the trees have the same trunk radius and are arranged such that no matter where you stand and where you look, you cannot see further than some distance which depends on the trunk radius. If there is a limit to how far you can see for every trunk radius and the forest has finite density, then it will be called a dense forest. 

A set of points $X\subset\bR^d$ has finite density if
\[
\lim_{T\to\infty} \frac{\#(X\cap B(0,T))}{T^d}<\infty
\]
where $\#(\cdot)$ is the cardinality of a set and $B(x,T)$ is the ball centred at $x$ of radius $T$.
The precise definition of a dense forest is given below.

\begin{dfn}
	Given a set of points $F\subset\bR^d$, $x\in\bR^d$,$v\in S^{d-1}$, and $\ep>0$, the visibility in $F$ at scale $\ep$ from $x$ in direction $v$ is the smallest $t\in\bR_{\geq 0}$ such that
	\[
	\n{x+tv-y}\leq \ep
	\]
	for some $y\in F$.
\end{dfn}
Here, $\n{\cdot}$ is the Euclidean norm. When they are unambiguous, the parameters in the above notation will be suppressed.
\begin{dfn}
	A set of points $F\subset\bR^d$ is a dense forest if it has finite density and there exists a visibility function $V:\bR_{>0}\to\bR_{\geq 0}$ such that for any $x\in\bR^d$, $v\in S^{d-1}$, and $\ep>0$, the visibility is at most $V(\ep)$.
\end{dfn} 
\noindent The points of $F$ can be thought of as the centres of the trees in the forest.

Does such a forest exist? The question can be traced back quite far. The problem can be connected to a problem posed by P\'olya, in \cite{MR0465631}(problem 239, page 151) for instance, often called P\'olya's Orchard Problem. If regularly spaced trees are placed on a circle, how thick must their trunks be to completely block the view from the centre?

The question in its current form was asked---and partially answered---by Bishop in \cite{MR2900167} by giving a construction credited to Peres. In fact, the version of the question in the paper was slightly different than the one given here; it asked if the forest could be made uniformly discrete---that is, there was a minimum distance by which any two tree centres must be separated. In the end, the construction was a union of two uniformly discrete sets and was given in $\bR^2$. The visibility in that initial forest was $O\left(\ep^{-4}\right)$. The problem has since attracted some attention. An improvement was given by Alon in \cite{MR3816056}. The forest was also given in $\bR^2$, with visibility $\nicefrac{2^{C\sqrt{\log\nicefrac{1}{\ep}}}}{\ep}$ for some constant $C$. It was also stated in the conclusion that the construction generalises to higher dimensions. The forest was constructed as an infinite intersection of sets of random point placements using the Lov\'asz local lemma and a compactness argument. It has the added benefit of being uniformly discrete.

A dense forest in $\bR^d$ with a visibility in $O\left(\ep^{-2(d-1)-\eta}\right)$ for any $\eta>0$ was given by Adiceam in \cite{MR3556427}. It was obtained as the union of lattices with their points shifted by a sequence which satisfies certain growth requirements.  
A further improved forest, with an explicit construction, was given by Tsokanos in \cite{MR4459135}. It was constructed by placing regularly spaced `curtains' of points in each coordinate direction which obstruct the visibility for different scales of $\ep$. The visibility bound there is in $O\left(\ep^{-(d-1)}\log(\ep^{-1})\log\log(\ep^{-1})^{1+\eta}\right)$ for any $\eta>0$. Explicit uniformly discrete dense forests, both with and without explicit visibility bounds, were also given by Solomon and Weiss in \cite{MR3581810} and by Adiceam, Solomon, and Weiss in \cite{MR4400945}. A scenic hike through dense forests and related problems is given in \cite{MR4420864}.

The visibility functions examined here will be of the form $\ep^{-(d-1)}E(\ep^{-1})$ where $E$ is an error term.
The main goal is then to minimise the growth rate of this term. The lower bound for the error term, for an `optimal' dense forest, has $E(\ep)\in O(1)$. It is not known if this lower bound is attainable by a dense forest. The relation of this theoretical ideal forest with other problems about discrepancy and their applications are discussed in the conclusion.

Below is the main result of this paper, a modified probabilistic construction of a dense forest with a logarithmic error term.

\begin{thm}\label{main}
	Let $d\geq 2$. There exists a dense forest in $\bR^d$ with visibility 
	\[V(\ep)\in O\left(\ep^{-(d-1)}\ln\ep^{-1}\right).\]
\end{thm}

%This is the currently the slowest growing visibility function known for a forest. 
It is obtained by adding points to a random set obtained from a Poisson point process in places where the visibility is estimated to be too large. The expected density is calculated and shown to be finite.

\noindent\textbf{Acknowledgements}

\noindent The author would like to thank Almut Burchard for many (many) helpful comments and suggestions. Part of the work on this paper was completed during a trimester program at the Hausdorff Research Institute for Mathematics.

\section{Proofs}

The forest will be constructed by first placing down points randomly and then adding points where the visibility exceeds the desired amount. In this case, the desired amount is $V(\ep)=\ep^{-(d-1)}E(\ep^{-1})$, where $E$ will be assumed to be non decreasing as $\ep\to 0$ and $\ep^{-1}\geq E(\ep^{-1})\geq\ep$ for $\ep\leq\frac{1}{2}$ and will be chosen later. The number of additional points placed will be small enough such that the resulting set will have finite density.

The following lemma will be used to help count regions where the visibility is too large. A box will refer to a product of intervals.

\begin{lma}\label{test box lemma}
	Let $E$ be as above. For any unit cube $C\subset\bR^d$ and any $0<\ep\leq\frac{1}{2}$, there exists a finite set of boxes with $d-1$ sides of length $\frac{\ep}{4\sqrt{d}}$ and one side of length $\frac{1}{4\sqrt{d}}\ep^{-(d-1)}E(\ep^{-1})$ of cardinality 
	\[ \left(\frac{4\sqrt{d}}{\ep}\right)^d\left(\frac{\pi}{2\ep^{d+1}}\right)^{d-1}\]
	such that any box with $d-1$ sides of length $2\ep$ and one side of length $\ep^{-(d-1)}E(\ep^{-1})$ with centre in $C$ entirely contains one of the smaller boxes.
\end{lma}

\begin{proof}[Proof of Lemma \ref{test box lemma}]
	The boxes with the smaller size described above will be called test boxes.
	Let $C$ be a unit cube in $\bR^d$. The cube may be assumed to be axis aligned. Consider a $\frac{\ep}{4\sqrt{d}}$ axis aligned grid of points on the cube. It will be used to position the test boxes in $C$. Multiple copies of the test boxes will be placed centred at each grid point to account for all possible rotations of the larger boxes. %A test box may be non uniquely identified with the position of its long axis in spherical coordinates about its centre and its rotation about its long axis. The rotations by $\theta$ below specify the angular components of the coordinates.
	
	For each grid point, an axis aligned test box with its long axis parallel to the $x_1$ axis is taken. Next, for each grid point, copies of the test box centred at it and rotated by all positive integer multiples of $\theta=2\ep^{d+1}$ such that $k\theta\leq \pi$ along all axis parallel axes, except the long axis, are added.
	There are
	\[\left(\frac{4\sqrt{d}}{\ep}\right)^d\]
	grid points and 
	\[\left(\frac{\pi}{2\ep^{d+1}}\right)^{d-1}\]
	test boxes centred at each grid point.
	
	It remains to show that each large box contains a test box. Any large box with its centre in $C$ contains a grid point, also in $C$, within distance at most $\frac{\ep}{2}$ of its centre and of distance at least $\frac{\ep}{2\sqrt{d}}$ from its boundary. Due to the $\frac{1}{\sqrt{d}}$ factor, it can be assumed that the large box is axis aligned. Call that grid point $G$. A copy of the large box scaled down to the size of a test box is placed centred at $G$. Let $P$ be the point on the border of the large box with minimal distance to $G$. There is a maximal angle, $\vp$, by which the scaled down box may be rotated toward the face on which $P$ lies while remaining inside the large box. Let $I$ be the point of intersection of the long axis of the rotated scaled down box with the large box (see Figure 1). \\
	Consider the right triangle formed by the points $G$, $P$, and $I$. The length of $GP$ is at least $\frac{\ep}{2\sqrt{d}}$ and the length of $GI$ is at most $\frac{1}{4\sqrt{d}}\ep^{-(d-1)}E(\ep^{-1})$. The angle $\vp$ is $GIP$. The angle $\theta$ is at most $\vp$ since
	\begin{align*}
		\vp\geq\sin\vp
		&\geq\frac{\frac{\ep}{2\sqrt{d}}}{\frac{1}{4\sqrt{d}}\ep^{-(d-1)}E(\ep^{-1})}\\
		&=\frac{2\ep^d}{E(\ep^{-1})}\\
		&\geq2\ep^{d+1}\\
		&=\theta.
		\end{align*}
	So, the scaled down box may be rotated by an angle of $\theta$ about the $d-1$ short axes and by any angle about its long axis while remaining inside the large box.
		
\end{proof}

\begin{figure}\label{rot}
	\begin{center}
		\begin{tikzpicture}[scale=0.53]
			
			\coordinate (a) at (16.5,11.5);
			\coordinate (b) at (12,7);
			\coordinate (c) at (17.75,9);
			
			\draw (0,5.2) to (9.8,15);
			\draw (9.8,0) to (20,10.2);
			
			\draw [draw=black,rotate around={45:(12,7)}] (7,6) rectangle (17,8);
			\draw [draw=red!70,rotate around={15:(12,7)}] (7,6) rectangle (17,8);
			
			\draw [dash dot,rotate around={45:(12,7)}] (12,7) -- (17,7) coordinate(original);
			%\draw [draw=red,dash dot,rotate around={15:(12,7)}] (12,7) -- (17,7) coordinate(moved);
			\draw [draw=red!80,dash dot,rotate around={15:(12,7)}] (12,7) -- (18.8,7) coordinate(moved);
			\draw [dotted,rotate around={-45:(12,7)}] (12,7) -- (15.5,7) coordinate(original);
			
			\pic[draw, <-, "$\theta$", angle eccentricity=1.5] {angle=c--b--a};
			\draw (12,7) node[below]{$G$} circle (0);
			\draw [rotate around={15:(12,7)}] (18.8,7) node[below]{$I$} circle (0);
			\draw [dotted,rotate around={-45:(12,7)}] (15.5,7) node[below]{$P$} circle (0);
			
			\foreach \x in {0,2,4,6,8,10,12,14,16,18,20}
			\foreach \y in {1,3,5,7,9,11,13,15}
			{
				\draw [fill] (\x,\y) circle (.035cm);
			}
		\end{tikzpicture}
		\caption{A sketch of a test box inside a large box along with its maximum rotation in one axis(in red).}
	\end{center}
\end{figure}

\begin{proof}[Proof of Theorem \ref{main}]
	To begin, take a point set $F\subset\bR^d$ obtained from a Poisson point process on $\bR^d$ with density $\lambda$, where $\lambda$ depends on the dimension and will be chosen later. The visibility in $F$ will be unbounded, since the Poisson process on $\bR^d$ will produce arbitrarily large gaps between points.
	The rest of the proof is concerned with adding points to $F$ to remove those gaps while maintaining finite density. This is done by discretising over scales of $\ep$ and, for each scale, adding points to reduce the visibility at that scale to at most $V(\ep)$. The sequence $\ep_k=2^{-k},k\geq 2$ will be used for the scales.
	
	For a given $\ep$, the probability that the visibility in $F$ from some point $x\in\bR^d$ in some direction $v\in S^{d-1}$ is greater than $V(\ep)$ is the probability that there are no points in $F$ within distance $\ep$ of the line segment $\{x+tv: 0\leq t\leq V(\ep)\}$. By John's theorem \cite{MR1153987}, up to a dimension dependent constant, this is equal to the probability of there being an empty box with $d-1$ sides of length $2\ep$ and one side of length $V(\ep)$ with its base centred at $x$ and extending in direction $v$. To limit the visibility in $F$, a point must be added to all such empty boxes. This is accomplished by constructing a set of test boxes for each scale of $\ep$. The test boxes are chosen so that each of the larger boxes contains a test box. That way, a point is placed in the centre of each empty test box ensures that none of the original boxes are empty.
	
	Since the density of the Poisson process is fixed, it is enough to restrict the calculations to a unit cube. At scale $k$, the boxes that will be examined are those with $d-1$ sides of length $2\ep_k$ and one of length $V(\ep_k)$ with their centres in a fixed unit cube $C$. By \cref{test box lemma}, there is a set of test boxes with $d-1$ sides of length $\frac{\ep}{4\sqrt{d}}$ and one side of length $\frac{1}{4\sqrt{d}}\ep_k^{-(d-1)}E(\ep_k^{-1})$ of cardinality 
	\[\left(\frac{4\sqrt{d}}{ \ep_k}\right)^d\left(\frac{\pi}{2\ep_k^{d+1}}\right)^{d-1}\]
	such that each of the large boxes with their centre in $C$ entirely contains one of the test boxes. The probability of such a test box being empty is 
	\[e^{-\frac{\lambda }{(4\sqrt{d})^d}E(\ep_k^{-1})}.\]
	 So, the expected number of points added at scale $k$ is 
	\[
	N_k=\left(\frac{4\sqrt{d}}{ \ep_k}\right)^d\left(\frac{\pi}{2\ep_k^{d+1}}\right)^{d-1}e^{-\frac{\lambda }{(4\sqrt{d})^d}E(\ep_k^{-1})}.
	\]
	It should be noted that the points added for $\ep_k$ will also reduce the visibility for all $\ep$ with $\ep_k\leq \ep<\ep_{k-1}$.
	
	Summing over all the scales, the total amount---in expectation---of points that were added is
	\begin{align*}
		\sum_{k=2}^\infty N_k
		&=\sum_{k=2}^\infty \left(\frac{4\sqrt{d}}{ 2^{-k}}\right)^d\left(\frac{\pi}{2(2^{-k})^{d+1}}\right)^{d-1}e^{-\frac{\lambda }{(4\sqrt{d})^d}E(2^{k})}\\
		&\lesssim\sum_{k=2}^\infty 2^{dk+(d+1)(d-1)k}e^{-\frac{\lambda }{(4\sqrt{d})^d}E(2^k)}\\
		&\leq\sum_{k=2}^\infty 2^{2(d+1)^2k}e^{-\frac{\lambda }{(4\sqrt{d})^d}E(2^k)}.
	\end{align*}
	
	Let $E(\ep)=\ln\ep$. The two main terms in the above sum are exponentials of the same order. The above sum is then finite for sufficiently large, depending on the dimension, $\lambda$. Therefore, the expected number of points added to each unit cube is finite and the density remains finite.
	
\end{proof}

Since any face of the unit cube in $\bR^d$ may be tiled by $2^{k(d-1)}$ cubes in $\bR^{d-1}$ of side length $2^{-k}$, there are at least $2^{k(d-1)}$ large boxes with disjoint interiors being examined. So, the number of test boxes needed at scale $\ep_k$ is at least $2^{k(d-1)}$. Therefore, the bounds on visibility cannot be improved by improving the approximation in Lemma \ref{test box lemma}.

\section{Conclusion}

Danzer's problem asks if there exists a set of finite density in $\bR^d$ which intersects every convex set of a fixed volume---a Danzer set. Relaxations of Danzer's problem, obtained by modifying one of the above requirements, may be used to help investigate the existence of a Danzer set. Instead of all convex sets of a fixed volume, again by John's theorem, it is enough to look at all boxes of a fixed volume. By restricting to boxes that have $d-1$ sides of length $\ep$ and one of length $\ep^{-(d-1)}E(\ep^{-1})$, where $E$ grows as $\ep\to 0$, one obtains the problem of finding a dense forest. So, an optimal dense forest would be one where all the boxes at different scales of $\ep$ have the same volume. That is, where the error term is constant. However, the existence of an optimal dense forest would not necessarily imply the existence of a Danzer set as, in dimension greater than $2$, a dense forest is not necessarily a Danzer set.

The other approach is to relax the finite density assumption and allow it to grow as one moves away from the origin. The goal here is then to minimise the growth of the density while still satisfying Danzer's property. Denoting by $g(T)$ the number of points of a set in the ball or radius $T$ centred at the origin, the optimal result would be $g(T)\in O(T^d)$. The current best result is given by Solomon and Weiss in \cite{MR3581810}. It is a set satisfying Danzer's property with $g(T)\in O(T^d\log T)$. They accomplished this by reducing the problem to an equivalent combinatorial problem about the existence of an $\ep$-net on the range space of boxes in a cube. It should be noted that their construction is random. In the conclusion of \cite{MR1397998}, it is conjectured that the lower bound obtained for the equivalent problem is the best possible, implying a negative answer to Danzer's problem.

If the range space mentioned above is restricted to include only axis aligned boxes, one  obtains the problem of minimal dispersion. Precisely, how many points, in terms of $\ep$ and dimension, must be placed in the unit cube in $\bR^d$ such that the largest empty axis aligned box has volume at most $\ep$? An overview of results in the field is given in the introduction of \cite{MR4438171}. The paper, like this one, used a set of test boxes to determine if all full sized boxes had a point in them. In terms of just $\ep$ dependence, sets with a dispersion of order $\frac{1}{\ep}$ have been shown to exist. This is, in part, possible due to the lack of rotation. This is consistent with the results about Danzer sets. Simmons and Solomon constructed in \cite{MR3477090} a set in $\bR^d$ of finite density which intersects every axis aligned box of unit volume in $\bR^d$.

All the above problems measure, in various senses, how well distributed a set of points is. Such sets are useful as they have many applications in numerical analysis. They are used in optimisation, numerical integration, and range query problems among others. 

Two features of the method used in the proof of the main result are noted here. The first is that the actual shape of the boxes does not play much of a role in the main argument. The only important part is that their volume grows sufficiently fast as $\ep\to0$ to guarantee that the sum converges. This approach to finding a dense forest is more evocative of Danzer's problem than some earlier constructions, which work more closely with the definition and rely more on the specific shape of the box.
The second is that most of the work of intersecting the boxes with a point is done by the random process. By optimising the parameter, $\lambda$, of the Poisson process, the number of points that need to be added afterwards can be made relatively small. So, while a random set obtained from a Poisson point process cannot be a dense forest, the dense forest constructed above is not too different from a random set. 

Since in dimension greater than $2$ a dense forest is not necessarily a Danzer set, the conjecture in \cite{MR1397998} does not entirely apply to dense forests. Since the growth rate of the visibility presented here is slower than $O(\ep^{-(d-1)-\eta})$ for any $\eta>0$, there exist forests with a visibility growth rate arbitrarily close to $\ep^{-(d-1)}$. So, an optimal dense forest is a sort of critical point for the existence of such sets. If one does not exist, the goal of minimising the growth of the error term remains. Most forests with explicit visibility bounds constructed so far, including the one here, have examined different scales of $\ep$ separately. The boxes for different scales of $\ep$ overlap and them containing a point inside are not independent events. There should be something to gain by reducing the redundancy.

The main focus in this paper was trying to minimise the growth of visibility in terms of $\ep$. Since the visibility bounds here are given in big-O notation, dimension dependent constants---which are rather large---have been downplayed. Relating this to the problem of minimal dispersion, what bounds can be found for the visibility if it is viewed as a function of both $\ep$ and dimension?

\end{document}